\documentclass[12pt,reqno]{amsart}

\usepackage[latin1]{inputenc}
\usepackage{amsmath}
\usepackage{amsfonts}
\usepackage{amssymb}
\usepackage{graphics}
\usepackage{enumerate}
\usepackage{amssymb,amsmath,amsthm,amscd,latexsym,verbatim,graphicx,amsfonts}

\usepackage{amscd}
\usepackage{amsmath}
\usepackage{amssymb}
\usepackage{amsthm}
\usepackage{latexsym}
\usepackage{verbatim}

\theoremstyle{plain}
\newtheorem{theorem}{Theorem}[section]

\newtheorem{lemma}[theorem]{Lemma}

\theoremstyle{definition}

\theoremstyle{remark}
\newtheorem{rem}[theorem]{Remark}

\newcommand{\nri}{n\rightarrow\infty}
\newcommand{\kri}{k\rightarrow\infty}
\newcommand{\tri}{t\rightarrow\infty}

\newcommand{\bbC}{\mathbb{C}}

\newcommand{\bbD}{\mathbb{D}}

\newcommand{\bbN}{\mathbb{N}}

\newcommand{\mca}{\mathcal{A}}
\newcommand{\mcj}{\mathcal{J}}

\newcommand{\mcl}{\mathcal{L}}
\newcommand{\mcm}{\mathcal{M}}
\newcommand{\mcr}{\mathcal{R}}

\newcommand{\eitheta}{e^{i\theta}}

\DeclareMathOperator*{\supp}{supp}

\title[]{Hyponormal Toeplitz Operators on Weighted Bergman Spaces}
\author[]{Trieu Le $\&$ Brian Simanek}
\date{}

\begin{document}
\maketitle

\begin{abstract}
We consider operators acting on a Hilbert space that can be written as the sum of a shift and a diagonal operator and determine when the operator is hyponormal.  The condition is presented in terms of the norm of an explicit block Jacobi matrix.  We apply this result to the Toeplitz operator with symbol ${z^n+c|z|^s}$ acting on certain weighted Bergman spaces and determine for what values of the constant $c$ this operator is hyponormal.  
\end{abstract}

\vspace{4mm}


\vspace{2mm}


\vspace{2mm}

\normalsize

\section{Introduction}\label{Intro}

A bounded operator $T$ acting on a Hilbert space is said to be \textit{hyponormal} if $[T^*,T]\geq0$, where $T^*$ denotes the adjoint of $T$.  One motivation for studying such operators comes from Putnam's inequality (see \cite[Theorem 1]{Putnam}), which says that hyponormal operators satisfy
\[
\|[T^*,T]\|\leq\frac{|\sigma(T)|_2}{\pi}
\]
where $\sigma(T)$ is the spectrum of $T$ and $|\cdot|_2$ denotes the two-dimensional area.

We will be interested in operators that can be written as a shift plus a diagonal operator.  Specifically, we will call an operator $M$ acting on a Hilbert space $H$ a \textit{weighted shift of multiplicity $n$} if there exists an orthonormal basis $\{e_k\}_{k=0}^{\infty}$ for $H$ and a sequence $\{a_k\}_{k=0}^{\infty}$ of complex numbers such that $Me_k=a_ke_{k+n}$.  Similarly, we will call an operator $M$ acting on a Hilbert space $H$ a \textit{diagonal operator} if there exists an orthonormal basis $\{e_k\}_{k=0}^{\infty}$ for $H$ and a sequence $\{a_k\}_{k=0}^{\infty}$ of complex numbers such that $Me_k=a_ke_{k}$.

Determining if a particular operator is hyponormal can be a challenging problem.  There are several classification results in the literature; one of the more famous being a 1988 result of Cowen that characterizes all hyponormal Toeplitz operators of a certain form acting on the Hardy space $H^2$ (see \cite{Cowen,Cowen2}).  The goal of this paper is to characterize all hyponormal Toeplitz operators of a certain form acting on weighted Bergman spaces of the unit disk.  The characterization will be given in terms of an explicit block Jacobi matrix, which we now discuss.

Block Jacobi matrices are matrices of the form
\[
\mcm=\begin{pmatrix}
B_1 & A_1 & 0 & \cdots & \cdots\\
A_1^* & B_2 & A_2 & 0 & \cdots\\
0 & A_2^* & B_3 & A_3 & \ddots\\
\vdots & \vdots & \ddots & \ddots & \ddots
\end{pmatrix}
\]
where each $A_m$ and $B_m$ is a $k\times k$ matrix for some fixed $k\in\bbN$ with $B_m=B_m^*$ and $\det(A_m)\neq0$ for all $m\in\bbN$.  An extensive introduction to the theory and applications of these operators can be found in \cite{DPS}.
In the context of our problem, a block Jacobi matrix is a symmetric operator that is densely defined on $\ell^2(\bbN_0)$ (where $\bbN_0=\bbN\cup\{0\}$) and its spectrum is a compact subset of the real line if the operator is bounded.  While the spectrum of such an operator is in general difficult to compute, one can verify that if $B_m\equiv0$ and $A_m\equiv I_{k\times k}$ for all $m\in\bbN$, then the spectrum of the corresponding block Jacobi matrix is $[-2,2]$.  In particular, the norm of this operator is $2$.

In the next section, we will state and prove our main result, which is a complete description of the complex numbers $c$ for which the operator $T+cD$ is hyponormal, where $T$ is a hyponormal bounded weighted shift and $D$ is a bounded diagonal operator (with respect to the same basis).    In Section \ref{app} we will discuss applications of the main result to Toeplitz operators of a particular form acting on weighted Bergman spaces of the unit disk.  We will see that our results allow us to significantly generalize \cite[Theorem 2]{SimHypo}.



\section{Main Result}\label{Main}

\begin{theorem}\label{better1}
Assume there is a basis $\{e_k\}_{k=0}^{\infty}$ of $H$ so that $T:H\rightarrow H$ is a hyponormal bounded weighted shift of multiplicity $n$ such that $\ker([T^*,T])=\{0\}$ and $D:H\rightarrow H$ is a bounded diagonal operator.  If $c\in\bbC$, then the operator $T+cD$ is hyponormal if and only if $\|cJ\|\leq1$, where
\[
J=[T^*,T]^{-1/2}([T^*,D]+[D^*,T])[T^*,T]^{-1/2}.
\]
\end{theorem}

\begin{proof}
Since $D^*D=DD^*$, we have
\[
[T^*+\bar{c}D^*,T+cD]=[T^*,T]+c[T^*,D]+\bar{c}[D^*,T].
\]
It follows that $V:=T+cD$ is hyponormal if and only if
\[
[T^*,T]+c[T^*,D]+\bar{c}[D^*,T]\geq0,
\]
which is equivalent to
\begin{equation}\label{Ibig}
I\geq \frac{1}{2}\left(cA^*+\bar{c}A\right),
\end{equation}
where $A=2[T^*,T]^{-1/2}[T,D^*][T^*,T]^{-1/2}$.  Note that $A$ is a weighted shift of multiplicity $n$ as well.  Notice that condition \eqref{Ibig} can be restated as
\[
\mbox{Re}[\langle\bar{c}Af,f\rangle]\leq1
\]
for all $f\in H$ with $\|f\|=1$.  In other words, equation \eqref{Ibig} is equivalent to the condition that $\{\mbox{Re}[\bar{c}\lambda]:\lambda\in W(A)\}\subseteq(-\infty,1]$.  Since $A$ is a weighted shift and $J=\mbox{Re}[-A]$, it follows from \cite[Proposition 16]{Shields} that
\[
\{\lambda:|\lambda|<\|J\|\}\subseteq W(A)\subseteq\{\lambda:|\lambda|\leq\|J\|\}
\]
and hence condition \eqref{Ibig} is satisfied if and only if $\|cJ\|\leq1$.
\end{proof}

\begin{rem}
The matrix $J$ from Theorem \ref{better1} is a block Jacobi matrix when expressed using the basis $\{e_k\}_{k=0}^{\infty}$.
\end{rem}

\section{Applications}\label{app}

\subsection{Weighted Bergman Spaces}\label{wberg}
Let $\mu$ be a probability measure on the interval $[0,1]$ with $1\in\supp(\mu)$ and $\mu(\{1\})=0$.  Using $\mu$, define the measure $\nu$ on the open unit disk $\bbD$ by $d\nu(r\eitheta)=d\mu(r)\times\frac{d\theta}{2\pi}$.  Let $\mca^2_{\nu}(\bbD)$ denote the weighted Bergman space of the unit disk defined by
\[
\mca^2_{\nu}(\bbD)=\left\{f:\int_{\bbD}|f(z)|^2d\nu(z)<\infty,\, f\mathrm{\, is\, analytic\, in\, }\bbD\right\}.
\]
We equip $\mca^2_{\nu}(\bbD)$ with the inner product
\[
\left\langle f,g\right\rangle_{\nu}=\int_{\bbD}f(z)\overline{g(z)}\,d\nu(z).
\]
Notice that the rotation invariance of the measure $\nu$ means the monomials $\{z^n\}_{n=0}^{\infty}$ are an orthogonal set in $L^2(\bbD,d\nu)$ and $\mca_{\nu}^2(\bbD)$.  It is a standard fact that $\mca_{\nu}^2(\bbD)$ is a reproducing kernel Hilbert space.  Let us define the set $\{\gamma_t\}_{t\in[0,\infty)}$ by
\[
\gamma_t:=\int_{\bbD}|z|^td\nu(z)=\int_{[0,1]}x^td\mu(x).
\]
Since $1\in\supp(\mu)$, the sequence $\{\gamma_n\}_{n\in\bbN}$ decays subexponentially as $\nri$, meaning for all $t>0$ it holds that $\gamma_{m+t}/\gamma_m\rightarrow1$ as $m\rightarrow\infty$.  Since $\mu(\{1\})=0$, we know $\gamma_t$ approaches $0$ as $\tri$.  With this notation it is true that
\[
\mca^2_{\nu}(\bbD)=\left\{f(z)=\displaystyle\sum_{n=0}^{\infty}a_nz^n:\displaystyle\sum_{n=0}^{\infty}|a_n|^2\gamma_{2n}<\infty\right\}
\]
and the inner product becomes
\[
\left\langle\displaystyle\sum_{n=0}^{\infty}a_nz^n,\displaystyle\sum_{n=0}^{\infty}b_nz^n\right\rangle_{\nu}=\displaystyle\sum_{n=0}^{\infty}a_n\bar{b}_n\gamma_{2n}.
\]
Of particular interest is the case when
\[
d\nu(z)=(\beta+1)(1-|z|^2)^{\beta}dA
\]
for some $\beta\in(-1,\infty)$, where $dA$ is normalized area measure on $\bbD$ (see \cite{BKLSS,HwangLee,HLP,LuLiu,LuShi}).  Notice that when $\beta=0$, the space $\mca_\nu(\bbD)$ is just the usual Bergman space of the unit disk.

If $\varphi\in L^{\infty}(\bbD)$, then we define the operator $T_{\varphi}:\mca^2_{\nu}(\bbD)\rightarrow \mca^2_{\nu}(\bbD)$ with symbol $\varphi$ by
\[
T_{\varphi}(f)=P_{\nu}(\varphi f),
\]
where $P_{\nu}$ denotes the orthogonal projection to $\mca^2_{\nu}(\bbD)$ in $L^2(\bbD,d\nu)$.  There is an extensive literature aimed at characterizing those symbols $\varphi$ for which the corresponding operator $T_{\varphi}$ is hyponormal, much of which focuses on the special case of the classical Bergman space of the unit disk (see \cite{AC,CC,FL,Hwang,Hwang2,HwangLee,HLP,LuLiu,LuShi,Sadraoui,SimHypo}).  The specific symbol we will focus on is $\varphi(z)=z^n+cq(|z|)$, where $n\in\bbN$, $c\in\bbC$, and $q$ is a bounded function on $[0,1]$ that is sufficiently regular at $1$.  The case $n=1$ and $q(r)=r^2$ in the classical Bergman space was considered in \cite{FL} while a broader range of $n$ and $q(r)=r^s$ for general $s\in(0,\infty)$ was previously considered in \cite{SimHypo}, where it was shown that hyponormality of $T_{\varphi}$ acting on the classical Bergman space implies $|C|\leq\frac{n}{s}$ and the converse holds if $s\geq2n$.  Now we can complete that result by invoking Theorem \ref{better1} to obtain necessary and sufficient conditions on the constant $c$ for $T_{\varphi}$ acting on any $\mca_{\nu}^2(\bbD)$ to be hyponormal.  As a result, we will recover the aforementioned result from \cite{SimHypo}.

Now we will state our main application of Theorem \ref{better1} and to do so, we will let $\mcl$ and $\mcr$ denote the left and right shift operators respectively on $\ell^2(\bbN_0)$ and define the orthonormal basis $\{e_k\}_{k=0}^{\infty}$ for $\mca_{\nu}^2(\bbD)$ by
\[
e_k=\gamma_{2k}^{-1/2}z^k.
\]

\begin{theorem}\label{bounded}
Suppose in Theorem \ref{better1} we take $T=T_{z^n}$ and $D=T_{q(|z|)}$ acting on the weighted Bergman space $\mca_\nu^2(\bbD)$, where $q:[0,1]\rightarrow\bbC$ is in $L^{\infty}([0,1],\mu)$ and is continuously differentiable in a neighborhood of $1$.  If $\alpha=q'(1)/(2n)$, then the operator $J$ from Theorem \ref{better1} is a compact perturbation of $\alpha\mcl^n+\bar{\alpha}\mcr^n$ in the basis $\{e_k\}_{k=0}^{\infty}$.
\end{theorem}

It is known that the operators $T$ and $D$ from Theorem \ref{bounded} satisfy the hypotheses of Theorem \ref{better1}.  Indeed, $T_f$ is bounded and hyponormal for any $f\in H^{\infty}(\bbD)$.  Furthermore, if $g\in\mbox{ker}([T_{z^n}^*,T_{z^n}])$, then
\[
\|T_{z^n}g\|^2=\int_{\bbD}|z^ng(z)|^2d\nu(z)=\int_{\bbD}|\bar{z}^ng(z)|^2d\nu(z)=\int_{\bbD}|P_{\nu}(\bar{z}^ng(z))|^2d\nu(z)=\|T_{z^n}^*g\|^2.
\]
This implies $\bar{z}^ng(z)\in\mca^2_{\nu}(\bbD)$, which implies $g=0$.  Our next lemma shows that $D$ from Theorem \ref{bounded} is a bounded diagonal operator with respect to $\{e_k\}_{k=0}^{\infty}$.  We can interpret this lemma as an adaptation of \cite[Lemma 1]{SimHypo}.

\begin{lemma}\label{bproj}
If $f\in L^{1}([0,1],\mu)$ and $k\in\bbN_0:=\bbN\cup\{0\}$, then
\[
P_{\nu}(f(|z|)z^k)=\frac{z^k}{\gamma_{2k}}\int_{[0,1]}f(x)x^{2k}d\mu(x).
\]
\end{lemma}

\begin{proof}
A calculation shows that
\[
\left\langle z^m,z^k\int_{[0,1]}f(x)x^{2k}d\mu(x)\right\rangle=\gamma_{2k}\left\langle z^m, f(|z|)z^k\right\rangle
\]
for every $m\in\bbN_0$, so the claim follows from the fact that polynomials are dense in $\mca_\nu^2(\bbD)$.
\end{proof}


We will also require the following elementary lemma.

\begin{lemma}\label{sube}
It holds that
\[
\lim_{\kri}\left(\int_{[0,1]^2}(xy)^{k}(x^{2n}-y^{2n})^2d\mu(x)d\mu(y)\right)^{1/k}=1
\]
\end{lemma}

\begin{proof}
It is clear that
\[
\limsup_{\kri}\left(\int_{[0,1]^2}(xy)^{k}(x^{2n}-y^{2n})^2d\mu(x)d\mu(y)\right)^{1/k}\leq1
\]
Also notice that if $\delta\in(0,1)$ is fixed, then
\begin{align*}
&\liminf_{\kri}\left(\int_{[0,1]^2}(xy)^{k}(x^{2n}-y^{2n})^2d\mu(x)d\mu(y)\right)^{1/k}\\
&\qquad\qquad\qquad\geq\liminf_{\kri}\left(\int_{[1-\delta,1]^2}(xy)^{k}(x^{2n}-y^{2n})^2d\mu(x)d\mu(y)\right)^{1/k}\\
&\qquad\qquad\qquad\geq(1-\delta)^2\liminf_{\kri}\left(\int_{[1-\delta,1]^2}(x^{2n}-y^{2n})^2d\mu(x)d\mu(y)\right)^{1/k}\\
&\qquad\qquad\qquad=(1-\delta)^2
\end{align*}
Sending $\delta\rightarrow0$ proves the lemma.
\end{proof}

\begin{proof}[Proof of Theorem \ref{bounded}]
We will calculate the matrix $A$ from the proof of Theorem \ref{better1} and show that it is a compact perturbation of $-2\bar{\alpha}\mcr^{n}$.  Since $\mcj=\mbox{Re}[-A]$, this will prove the desired result.

For ease of notation, let $h=\bar{q}$ and define
\[
\hat{h}(m)=\int_{[0,1]}h(x)x^md\mu(x).
\]
Elementary calculations that apply Lemma \ref{bproj} reveal that
\[
[T^*_{z^n},T_{z^n}]e_k=\begin{cases}
\frac{\gamma_{2k+2n}}{\gamma_{2k}}e_k\qquad\qquad\qquad\qquad & k\in\{0,1,\ldots,n-1\},\\
\left(\frac{\gamma_{2k+2n}}{\gamma_{2k}}-\frac{\gamma_{2k}}{\gamma_{2k-2n}}\right)e_k & k\in\{n,n+1,\ldots\}
\end{cases}
\]
and
\[
[T_{z^n},T_{q(|z|)}^*]e_k=[T_{z^n},T_{h(|z|)}]e_k=\sqrt{\frac{\gamma_{2k+2n}}{\gamma_{2k}}}\left(\frac{\hat{h}(2k)}{\gamma_{2k}}-\frac{\hat{h}(2k+2n)}{\gamma_{2k+2n}}\right)e_{k+n}
\]
Thus, for $k\geq n$ it holds that $Ae_k=\lambda_ke_{k+n}$, where
\begin{align}\label{jmodk}
\nonumber\lambda_k&=2\frac{\sqrt{\frac{\gamma_{2k+2n}}{\gamma_{2k}}}\left(\frac{\hat{h}(2k)}{\gamma_{2k}}-\frac{\hat{h}(2k+2n)}{\gamma_{2k+2n}}\right)}{\sqrt{\frac{\gamma_{2k+4n}}{\gamma_{2k+2n}}-\frac{\gamma_{2k+2n}}{\gamma_{2k}}}\sqrt{\frac{\gamma_{2k+2n}}{\gamma_{2k}}-\frac{\gamma_{2k}}{\gamma_{2k-2n}}}}\\
&=(2+o(1))\frac{\gamma_{2k+2n}\hat{h}(2k)-\gamma_{2k}\hat{h}(2k+2n)}{\sqrt{(\gamma_{2k+4n}\gamma_{2k}-\gamma_{2k+2n}^2)(\gamma_{2k+2n}\gamma_{2k-2n}-\gamma_{2k}^2)}},
\end{align}
as $\kri$, where we used the fact that for any $\eta>0$ it holds that
\[
\lim_{t\rightarrow\infty}\frac{\gamma_{t+\eta}}{\gamma_t}=1.
\]
Now we write
\begin{align*}
\gamma_{2k+4n}\gamma_{2k}-\gamma_{2k+2n}^2&=\int_{[0,1]^2}(x^{2k+4n}y^{2k}-x^{2k+2n}y^{2k+2n})d\mu(x)d\mu(y)\\
&=\int_{[0,1]^2}(xy)^{2k}x^{2n}(x^{2n}-y^{2n})d\mu(x)d\mu(y).
\end{align*}
Interchanging the roles of $x$ and $y$ and adding these expressions, we find
\begin{equation*}
\gamma_{2k+4n}\gamma_{2k}-\gamma_{2k+2n}^2=\frac{1}{2}\int_{[0,1]^2}(xy)^{2k}(x^{2n}-y^{2n})^2d\mu(x)d\mu(y).
\end{equation*}
Using similar reasoning on the other expressions in \eqref{jmodk}, we can rewrite the leading term of \eqref{jmodk} as
\begin{align}\label{lead}
\frac{2\int_{[0,1]^2}(xy)^{2k}(h(x)-h(y))(y^{2n}-x^{2n})d\mu(x)d\mu(y)}{\sqrt{\int_{[0,1]^2}(xy)^{2k}(x^{2n}-y^{2n})^2d\mu(x)d\mu(y)\cdot\int_{[0,1]^2}(xy)^{2k-2n}(x^{2n}-y^{2n})^2d\mu(x)d\mu(y)}}.
\end{align}
Lemma \ref{sube} tells us that the denominator in \eqref{lead} decays subexponentially as $\kri$, so any exponentially decaying perturbation of the numerator will not affect the limit.  We conclude that for any $\epsilon\in(0,1)$ it holds that
\[
\lim_{\kri}\lambda_k=\lim_{\kri}\frac{2\int_{[1-\epsilon,1]^2}(xy)^{2k}(h(x)-h(y))(y^{2n}-x^{2n})d\mu(x)d\mu(y)}{\sqrt{\int_{[0,1]^2}(xy)^{2k}(x^{2n}-y^{2n})^2d\mu(x)d\mu(y)\cdot\int_{[0,1]^2}(xy)^{2k-2n}(x^{2n}-y^{2n})^2d\mu(x)d\mu(y)}}
\]

Observe that if $r$ is sufficiently close to $1$, then
\begin{equation}\label{q1}
\lim_{(x,y)\rightarrow(r,r)}\frac{q(x)-q(y)}{x^{2n}-y^{2n}}=\lim_{(x,y)\rightarrow(1,1)}\frac{(x-y)^{-1}\int_y^xq'(t)dt}{2n(x-y)^{-1}\int_y^xt^{2n-1}dt}=\frac{q'(r)}{2nr^{2n-1}},
\end{equation}
where we used the continuity of $q'$ in a neighborhood of $1$ in the last step.  
The continuity of the right-hand side of \eqref{q1} near $1$ tells us that if $\delta>0$ is fixed, then we can choose $\epsilon>0$ so that
\[
\left|\frac{h(x)-h(y)}{x^{2n}-y^{2n}}-\bar{\alpha}\right|<\delta,\qquad\qquad\qquad(x,y)\in[1-\epsilon,1]^2,
\]
where we set $\frac{h(x)-h(y)}{x^{2n}-y^{2n}}=\frac{\overline{q'(x)}}{2nx^{2n-1}}$ when $x=y$.  Then
\begin{align*}
&\int_{[1-\epsilon,1]^2}(xy)^{2k}(h(x)-h(y))(y^{2n}-x^{2n})d\mu(x)d\mu(y)\\
&\qquad\qquad\qquad=-\int_{[1-\epsilon,1]^2}(xy)^{2k}\frac{h(x)-h(y)}{x^{2n}-y^{2n}}(x^{2n}-y^{2n})^2d\mu(x)d\mu(y)\\
&\qquad\qquad\qquad=-\bar{\alpha}\int_{[1-\epsilon,1]^2}(xy)^{2k}(x^{2n}-y^{2n})^2d\mu(x)d\mu(y)+E_k,
\end{align*}
where
\[
|E_k|<\delta\int_{[1-\epsilon,1]^2}(xy)^{2k}(x^{2n}-y^{2n})^2d\mu(x)d\mu(y).
\]
Thus, to evaluate $\lim_{\kri}\lambda_k$, it suffices to evaluate
\[
\lim_{\kri}\frac{\int_{[1-\epsilon,1]^2}(xy)^{2k}(x^{2n}-y^{2n})^2d\mu(x)d\mu(y)}{\sqrt{\int_{[0,1]^2}(xy)^{2k}(x^{2n}-y^{2n})^2d\mu(x)d\mu(y)\cdot\int_{[0,1]^2}(xy)^{2k-2n}(x^{2n}-y^{2n})^2d\mu(x)d\mu(y)}}
\]
Invoking Lemma \ref{sube} as above allows us to replace this limit by (where $\epsilon'>0$ is arbitrary)
\begin{align*}
&\lim_{\kri}\frac{\int_{[0,1]^2}(xy)^{2k}(x^{2n}-y^{2n})^2d\mu(x)d\mu(y)}{\sqrt{\int_{[0,1]^2}(xy)^{2k}(x^{2n}-y^{2n})^2d\mu(x)d\mu(y)\cdot\int_{[0,1]^2}(xy)^{2k-2n}(x^{2n}-y^{2n})^2d\mu(x)d\mu(y)}}\\
&\qquad\qquad\qquad\qquad=\lim_{\kri}\sqrt{\frac{\int_{[0,1]^2}(xy)^{2k}(x^{2n}-y^{2n})^2d\mu(x)d\mu(y)}{\int_{[0,1]^2}(xy)^{2k-2n}(x^{2n}-y^{2n})^2d\mu(x)d\mu(y)}}\\
&\qquad\qquad\qquad\qquad=\lim_{\kri}\sqrt{\frac{\int_{[1-\epsilon',1]^2}(xy)^{2k}(x^{2n}-y^{2n})^2d\mu(x)d\mu(y)}{\int_{[1-\epsilon',1]^2}(xy)^{2k-2n}(x^{2n}-y^{2n})^2d\mu(x)d\mu(y)}}
\end{align*}
by Lemma \ref{sube}.  This last limit is in the interval $[(1-\epsilon')^{2n},1]$ and since $\epsilon'>0$ was arbitrary, the limit must be $1$.  Since we chose $\delta>0$ to be arbitrary, it follows that
\[
\lim_{\kri}\lambda_k=-2\bar{\alpha}
\]
as desired.
\end{proof}

As an example of Theorem \ref{bounded}, consider the case when $q(t)=t^s$ with $0<s<\infty$.  In this case, the matrix $J$ from Theorem \ref{better1} is a compact perturbation of $\frac{s}{2n}(\mcl^n+\mcr^n)$.  Thus the essential spectrum of $J$ is $[-\frac{s}{n},\frac{s}{n}]$ and so $\|J\|\geq\frac{s}{n}$, which generalizes part of \cite[Theorem 2]{SimHypo}.  If we specialize further to the case when $\mu=2rdr$, then we can recover the rest of \cite[Theorem 2]{SimHypo}.  Indeed, in this case the measure $\nu$ is normalized area measure on $\bbD$ and $\gamma_{t}=2(t+2)^{-1}$ so one can calculate that $J$ is given by
\begin{align*}
&J_{n+k,k}=J_{k,n+k}=\begin{cases}
\frac{s\sqrt{(k+n+1)(k+2n+1)}}{2(k+1+s/2)(k+n+1+s/2)}\qquad\qquad\qquad&  k=0,\ldots,n-1\\
\,\\
\frac{s(k+1)\sqrt{(k+n+1)(k+2n+1)}}{2n(k+1+s/2)(k+n+1+s/2)}& k\geq n.
\end{cases}
\end{align*}
One can verify by hand that if $s\geq2n$, then each non-zero entry of $J$ is less than $\frac{s}{2n}$ and the non-zero diagonals approach $\frac{s}{2n}$ as we move along them.  This implies that when $s\geq2n$ the spectrum of $J$ is precisely $[-\frac{s}{n},\frac{s}{n}]$ and so Theorem \ref{better1} implies \cite[Theorem 2]{SimHypo}.

Theorem \ref{bounded} also provides a certain monotonicity in $|q'(1)|$ of $\|J\|$ for the matrix $J$ associated to $T_{z^n+q(|z|)}$ by Theorem \ref{better1}.  More precisely, as $|q'(1)|\rightarrow\infty$ it holds that $\|J\|\rightarrow\infty$.  With this observation, one can reason as in the proof of Theorem \ref{bounded} to conclude that $T_{z^n+c\sqrt{1-|z|^2}}$ is hyponormal if and only if $c=0$.


\vspace{10mm}

\noindent Trieu Le, Department of Mathematics and Statistics, The University of Toledo, Toledo, OH 43606

\bigskip

\noindent Brian Simanek, Department of Mathematics, Baylor University, Waco, TX 76798

\end{document}